\documentclass[10pt]{article}
\usepackage[active]{srcltx}
\usepackage[ansinew]{inputenc}
\usepackage{latexsym,amsmath,amssymb,amsthm,amsfonts}

\usepackage{fancybox}
\usepackage{epsfig}

\newtheorem{theorem}{Theorem}[section]
\newtheorem{lemma}[theorem]{Lemma}
\newtheorem{definition}[theorem]{Definition}

\newtheorem{proposition}[theorem]{Proposition}
\newtheorem{remark}[theorem]{Remark}
\numberwithin{equation}{section}

\begin{document}

\newcommand{\abs}[1]{\left\vert #1 \right\vert}

\newcommand{\set}[1]{\left\{ #1 \right\}}

\newcommand{\seq}[1]{\langle #1 \rangle}

\allowdisplaybreaks

\renewcommand{\thefootnote}{$\star$}

\title{On some Algebraic Properties for $q$-Meixner Multiple Orthogonal Polynomials of the First Kind}

\author{Jorge Arves\'u\thanks{The research of J. Arves\'u was partially supported by the
research grant MTM2012-36732-C03-01 of the Ministerio de Educaci\'on y Ciencia
 of Spain and mobility research grant from Fundaci\'on Caja Madrid.}\, and A.M. Ram\'{\i}rez-Aberasturis
\\Department of Mathematics, Universidad Carlos III de Madrid,\\
Avenida de la Universidad, 30, 28911, Legan\'es, Spain}
\date{\today }
\maketitle

\abstract{
We study a~new family of $q$-Meixner multiple orthogonal polynomials of the first kind. The discrete orthogonality conditions are considered over a non-uniform lattice with respect to 
different $q$-analogues of Pascal distributions. We address some algebraic properties, namely raising and lowering operators as well as Rodrigues-type.
Based on the explicit expressions for the raising and lowering operator a~high-order linear $q$-dif\/ference equation with polynomial coef\/f\/icients for the $q$-Meixner multiple orthogonal polynomials of the first kind is obtained. Finally, we obtain the nearest neighbor recurrence relation based on a purely algebraic approach.
}\\[2mm]
\\[2mm]

\renewcommand{\thefootnote}{\arabic{footnote}} \setcounter{footnote}{0}

\section{Introduction}

We begin by introducing the basic background materials. Let $\vec{\mu}=(\mu_{1},\ldots,\mu_{r})$ be a~vector of~$r$ positive discrete measures (with finite moments)
\begin{gather*}
\mu_{i}=\sum\limits_{k=0}^{N_{i}}\omega_{i,k}\delta_{x_{i,k}},
\qquad
\omega_{i,k}>0,
\qquad
x_{i,k}\in \mathbb{R},
\qquad
N_{i}\in \mathbb{N\cup}\{+\infty \},
\qquad
i=1,2,\ldots,r,
\label{dismeasure}
\end{gather*}
where $\delta_{x_{i,k}}$ denotes the Dirac delta function and $x_{i_{1},k}\neq x_{i_{2},k}$, $k=0,\ldots,N_{i}$,
whenever $i_{1}\neq i_{2}$. By 
$\vec{n}=(n_{1},\ldots,n_{r})\in\mathbb{N}^{r}$ we denote a~multi-index, where
$\mathbb{N}$ stands for the set of all nonnegative integers.
A~type~II discrete multiple orthogonal polynomial $P_{\vec{n}}(x)$, corresponding to the multi-index $\vec{n}$, is a~polynomial of
degree $\leq\vert \vec{n}\vert =n_{1}+\dots +n_{r}$ which satisf\/ies the orthogonality
conditions~\cite{arvesu_vanAssche}
\begin{gather}
\sum\limits_{k=0}^{N_{i}}P_{\vec{n}}(x_{i,k}) x_{i,k}^{j}\omega_{i,k} =0,
\qquad
j=0,\ldots,n_{i}-1,
\qquad
i=1,\ldots,r.
\label{OrthC}
\end{gather}
The orthogonality conditions~\eqref{OrthC} give a~linear system of $\vert\vec{n}\vert$
homogeneous equations for the $\vert\vec{n}\vert +1$ unknown coef\/f\/icients of $P_{\vec{n}}(x)$.
This polynomial solution $P_{\vec{n}}$ always exists.
We restrict our attention to a~unique solution (up to a~multiplicative factor) with $\deg P_{\vec{n}}(x)=|\vec{n}|$.
If this happen for every multi-index $\vec{n}$, we say that $\vec{n}$ is normal~\cite{Nikishin}.
If the above system of measures forms an $AT$ system~\cite{arvesu_vanAssche,Nikishin} then every multi-index is normal.
In this paper we will deal with such system of discrete measures.

In~\cite{arvesu_vanAssche} some type II discrete multiple orthogonal polynomials on the linear lattice
$x(s)=s$ were considered. Moreover, multiple Meixner polynomials (of first and second kind, respectively) were studied. Indeed, the monic multiple Meixner polynomials of the first kind~\cite{arvesu_vanAssche} $M_{\vec{n}}^{\vec{\alpha},\beta}(x)$, with multi-index
$\vec{n}\in \mathbb{N}^{r}$ and deg\-ree~$\vert \vec{n}\vert $, satisfy the following orthogonality conditions with dif\/ferent positive parameters $\alpha_{1},\ldots,\alpha_{r}\in(0,1)$ (indexed~by
$\vec{\alpha}=(\alpha_{1},\ldots,\alpha_{r})$) and $\beta>0$
\begin{gather*}
\sum\limits_{x=0}^{\infty}M_{\vec{n}}^{\vec{\alpha},\beta}(x)(-x)_{j}
\upsilon^{\alpha_{i},\beta}(x)=0,
\qquad
j=0,\ldots,n_{i}-1,
\qquad
i=1,\ldots,r,
\end{gather*}
where
\begin{gather*}
\upsilon^{\alpha_{i},\beta}(x)=
\begin{cases}
\dfrac{\Gamma (\beta+x)}{\Gamma (\beta)}\dfrac{\alpha_{i}^{x}}{\Gamma (x+1)}, & \text{if} \quad  x\in\mathbb{R}\setminus\left(\mathbb{Z}_{-}\cup\{-\beta,-\beta-1,\beta-2,\dots\}\right),
\\
0, & \text{otherwise,}
\end{cases}
\end{gather*}
and $(x)_j=(x)(x+1)\cdots(x+j-1)$, $(x)_0=1$, $j\geq 1$, denotes
the Pochhammer symbol. Notice that the multi-index $\vec{n}\in\mathbb{N}^{r}$ is normal whenever
$0<\alpha_{i}<1$, $i=1,2,\ldots,r$, and with all the $\alpha_{i}$ dif\/ferent (see \cite{arvesu_vanAssche}). 

Furthermore, it was found the following raising operators
\begin{gather}
\mathcal{R}_{\vec{n}}^{\alpha_{i},\beta}\big[M_{\vec{n}}^{\vec{\alpha},\beta}(x)\big ]
=-M_{\vec{n}+\vec{e}_{i}}^{\vec{\alpha},\beta-1}(x),
\qquad
i=1,\ldots,r,
\label{raisingx}
\end{gather}
where
\begin{gather*}
\mathcal{R}_{\vec{n}}^{\alpha_{i},\beta}
=\frac{\alpha_{i}\left(\beta-1\right)}{\left(1-\alpha_{i}\right)\upsilon^{\alpha_{i},\beta-1}(x)}\bigtriangledown\upsilon^{\alpha_{i},\beta}(x),
\end{gather*}
and $\bigtriangledown f(x)=f(x)-f(x-1)$ denotes the backward dif\/ference operator.
As a~consequence of~\eqref{raisingx} the following Rodrigues-type formula
\begin{gather}
M_{\vec{n}}^{\vec{\alpha},\beta}(x)
=(\beta)_{\vert\vec{n}\vert}\prod\limits_{j=1}^{r}\left(\frac{\alpha_{j}}{\alpha_{j}-1}\right)^{n_{j}}
 \frac{\Gamma(\beta)\Gamma(x+1)}{\Gamma(\beta+x)}\mathcal{M}_{\vec{n}}^{\vec{\alpha}}\left(\frac{\Gamma(\beta+\vert\vec{n}\vert +x)}{\Gamma(\beta+\vert\vec{n}\vert)\Gamma(x+1)}\right),
\label{Rodrigues-multi}
\end{gather}
where
\begin{gather*}
\mathcal{M}_{\vec{n}}^{\vec{\alpha}}=\prod\limits_{i=1}^{r}
\big(\alpha_{i}^{-x}\bigtriangledown^{n_{i}}\alpha_{i}^{x}\big),
\end{gather*}
was obtained.

Moreover, two important algebraic properties were deduced for multiple Meixner polynomials of the first kind~\cite{arvesu_vanAssche}, namely the
$(r+1)$-order linear dif\/ference equation~\cite{lee}
\begin{gather}
 \prod\limits_{i=1}^{r}\mathcal{R}_{\vec{n}}^{\alpha_{i},\beta}
\bigtriangleup M_{\vec{n}}^{\vec{\alpha},\beta}(x) 
+\sum\limits_{i=1}^{r}n_{i}\prod\limits_{\substack{j=1
\\
j\neq i}}^{r} \mathcal{R}_{\vec{n}}^{\alpha_{j},\beta}
M_{\vec{n}}^{\vec{\alpha},\beta}(x) =0,
\label{opdi-1}
\end{gather}
where $\bigtriangleup f(x)=f(x+1)-f(x)$,
and the recurrence relation~\cite{arvesu_vanAssche,Hane}
\begin{align}
xM_{\vec{n}}^{\vec{\alpha},\beta}(x) &=M_{\vec{n}+\vec{e}_{k} }^{\vec{\alpha},\beta}(x) +\left( (\beta+\vert\vec{n}\vert )\left(\frac{\alpha_{k}}{1-\alpha_{k}}\right)+\sum\limits_{i=1}^{r}\frac{n_{i}}{1-\alpha_{i}} \right) M_{\vec{n}}^{\vec{\alpha}\beta}(x)\notag\\
&+\sum\limits_{i=1}^{r}\frac{\alpha_{i}n_{i}\left(\beta +\vert\vec{n}\vert -1 \right)}{\left(\alpha_{i}-1\right)^{2}}M_{\vec{n}-\vec{e}_{i}}^{\vec{\alpha},\beta}(x),
\label{AR}
\end{align}
where the multi-index $\vec{e}_{i}$ is the standard~$r$ dimensional unit vector with the~$i$-th entry equals $1$ and $0$
otherwise.

The multiple Meixner polynomials of the first kind $M_{\vec{n},\beta}^{\vec{\alpha}}(x)$ are common eigenfunctions of the above
two linear dif\/ference operators of order $(r+1)$, given by~\eqref{opdi-1} and~\eqref{AR}, respectively.

In this paper we will introduce a $q$-analogue of such multiple orthogonal polynomials, i.e. when the component measures of $\vec{\mu}$ are dif\/ferent $q$-Poisson distributions and study the aforementioned algebraic properties. Our goal is to continue the recent investigations in ~\cite{arvesu-qHahn,arvesu-ramirez,Postelmans_Assche} for some families of $q$-multiple orthogonal polynomials regarding their algebraic properties. In~\cite{arvesu-esposito,lee,Assche_diff-eq} an $(r+1)$-order dif\/ference equation for some discrete multiple
orthogonal polynomials was obtained. Furthermore, the explicit expressions for the coef\/f\/icients of  $(r+2)$-term recurrence relations are a very important issue for the study of some type of asymptotic behaviors for discrete multiple orthogonal polynomials. In~\cite{apt-arv} the weak asymptotics was studied for multiple Meixner polynomials of the f\/irst and
second kind, respectively. The zero distribution of multiple Meixner polynomials was also studied. Another interesting fact involving the knowledge of the $(r+2)$-term recurrence relations is the attainment of a~Christof\/fel--Darboux kernel~\cite{daems_kuijlaars_kernel} among other applications, which plays important role in correlation kernel as in the unitary random matrix model with external source.

The content of this paper is as follows. In Section~\ref{q-meix} we def\/ine the $q$-Meixner multiple orthogonal polynomials  of the first kind. Moreover, we will prove that these multiple orthogonal polynomials can be explicitly expressed by  means of Rodrigues-type formula. This fact provides the background materials for the next Section~\ref{df-eq}, in which we obtain the  $(r+1)$-order $q$-dif\/ference equation, with polynomial coefficients on a~non-uniform lattice $x(s)$. Finally, in Section~\ref{rr} the nearest neighbor recurrence relation is obtained (an $(r+2)$-term recurrence relations). Explicit expressions for the recurrence coef\/f\/icients are given. The paper ends by summarising our findings in Section \ref{conclu}.

\section[$q$-Meixner multiple orthogonal polynomials of the first kind]{$\boldsymbol{q}$-Meixner multiple orthogonal polynomials of the first kind}\label{q-meix}

Aimed to define a new family of  $q$-multiple orthogonal polynomials~\cite{arvesu-qHahn} let us consider the following~$r$ positive discrete measures on $\mathbb{R}^{+}$,
\begin{gather}
\mu_{i}=\sum\limits_{s=0}^{\infty}\omega_{i}(k)\delta(k-s),
\qquad
\omega_{i}>0,
\qquad
i=1,2,\ldots,r.
\label{Measu}
\end{gather}
Here $\omega_{i}(s)=\upsilon_{q}^{\alpha_{i},\beta}(s) \bigtriangleup x(s-1/2)$, $x(s)=(q^{s}-1)/(q-1)$, and 
\begin{gather*}
\upsilon_{q}^{\alpha_{i},\beta}(s) =
\begin{cases}
\dfrac{\alpha_{i}^{s}}{\Gamma_{q}(s+1)}\dfrac{\Gamma_{q}(\beta + s)}{\Gamma_{q}(\beta)}, & \text{if}\quad s\in \mathbb{R}^{+}\cup \{0\},
\\
0, &\text{otherwise},
\end{cases}
\end{gather*}
where $0<\alpha_{i}<1$, $\beta>0$, $i=1,2,\ldots,r$, and with all the $\alpha_{i}$ dif\/ferent.
Recall that the $q$-Gamma function is given~by
\begin{gather*}
\Gamma_q(s) =
\begin{cases}
f(s;q)=(1-q)^{1-s} \dfrac{\prod\limits_{k\geq0} (1-q^{k+1})}{\prod\limits_{k\geq0} (1-q^{s+k})}, & 0<q<1,
\\
q^{\frac{(s-1)(s-2)}{2}}f\big(s;q^{-1}\big), & q>1.
\end{cases}
\label{q-gamma-clas}
\end{gather*}
See also~\cite{Gasper,nsu} for the above def\/inition of the $q$-Gamma function.

\begin{lemma}~\cite{arvesu-ramirez}
\label{le-1}
The system of functions
\begin{gather*}
 \alpha_1^s, x(s)\alpha_1^s,\ldots,x(s)^{n_1-1}\alpha_1^s,\ldots, \alpha_r^s,
x(s)\alpha_r^s,\ldots,x(s)^{n_r-1}\alpha_r^s,
\label{cheb}
\end{gather*}
with $\alpha_{i}>0$, $i=1,2,\ldots,r$, and $(\alpha_i/\alpha_j)\neq q^k$,
$k\in\mathbb{Z}$,
$i,j=1,\dots,r$,
$i\neq j$, forms a~Chebyshev system on $\mathbb{R}^{+}$ for every
$\vec{n}=(n_1,\dots,n_r)\in\mathbb{N}^{r}$.
\end{lemma}

As a~consequence of Lemma~\ref{le-1} the system of measures $\mu_1,\mu_2,\ldots,\mu_r$ given in~\eqref{Measu} forms an
AT system on $\mathbb{R}^{+}$. For this system of measures we will define a new family of $q$-multiple orthogonal polynomials.

Recall that the $q$-analogue of the Stirling polynomials denoted by $[s]_{q}^{(k)}$, is
a~polynomial of degree~$k$ in the variable $x(s)=(q^s-1)/(q-1)$, i.e.
\begin{gather*}
[s]_{q}^{(k)}=\prod\limits_{j=0}^{k-1}\frac{q^{s-j}-1}{q-1} =x(s) x(s-1) \cdots x(s-k+1)
\qquad\!
\text{for}
\quad
k>0,
\qquad\!
\text{and}
\qquad\!
[s]_{q}^{(0)}=1.
\end{gather*}
Moreover, $[s]_q^{(k)}=q^{-\binom{k}{2}}x^k(s)+\text{lower terms}=\mathcal{O}(q^{ks})$, where $\mathcal{O}(\cdot)$ stands for the big-O notation.
Observe that, when~$q$ goes to 1, the symbol $[s]_{q}^{(k)}$ converges to $(-1)^{k}(-s)_{k}$, where $(s)_{k}$
is the Pochhammer symbol.

\begin{definition}
A~polynomial $M_{q,\vec{n}}^{\vec{\alpha},\beta}(s)$, with multi-index $\vec{n}\in \mathbb{N}^{r}$ and degree $\vert
\vec{n}\vert$ that verif\/ies the orthogonality conditions
\begin{gather}
\sum\limits_{s=0}^{\infty}M_{q,\vec{n}}^{\vec{\alpha},\beta}(s) [s]_{q}^{(k)}\upsilon_{q}^{\alpha_{i},\beta}(s)
\bigtriangleup x(s-1/2) =0,
\qquad
0\leq k\leq n_{i}-1,
\qquad
i=1,\ldots,r,
\label{NOrthC}
\end{gather}
(with respect to the measures~\eqref{Measu}) is said to be the $q$-Meixner multiple orthogonal polynomial of the first kind.
\end{definition}

In the sequel, we will use the following dif\/ference operators
\begin{gather}
\Delta  \overset{\rm def}{=}\frac{\bigtriangleup} {\bigtriangleup x(s-1/2)},
\label{lower}
\\
\nabla^{n_{i}} =\underbrace{\nabla\cdots\nabla}_{n_{i}~\text{times}},\quad\mbox{where}\quad \nabla \overset{\rm def}{=} \frac{\bigtriangledown}{\bigtriangledown x(s+1/2)},
\label{nbackward}\notag
\end{gather}
and $\bigtriangledown x_{1}(s)\overset{\rm def}{=}\bigtriangledown x(s+1/2)= \bigtriangleup x(s-1/2)=q^{s-1/2}$.  Moreover, we will only consider monic $q$-Meixner multiple orthogonal polynomials of the first kind.

\begin{proposition}
There holds the following $q$-analogue of Rodrigues-type formula
\begin{gather}
M_{q,\vec{n}}^{\vec{\alpha},\beta}(s)
=\mathcal{K}_{q}^{\vec{n},\vec{\alpha},\beta}\frac{\Gamma_{q}(\beta)\Gamma_{q}(s+1)}{\Gamma_{q}(\beta+s)}\mathcal{M}_{q,\vec{n}}^{\vec{\alpha}}\left(\frac{\Gamma_{q}(\beta+\vert \vec{n} \vert+s)}{\Gamma_{q}(\beta+\vert \vec{n} \vert)\Gamma_{q}(s+1)}\right),
\label{RFormula}
\end{gather}
where
\begin{gather}
\mathcal{M}_{q,\vec{n}}^{\vec{\alpha}}=\prod\limits_{i=1}^{r}\mathcal{M}_{q,n_{i}}^{\alpha_{i}},
\qquad
\mathcal{M}_{q,n_{i}}^{\alpha_{i}}=(\alpha_{i})^{-s} \nabla^{n_{i}}(\alpha_{i}q^{n_i})^{s},
\label{Op-q-Meix}
\end{gather}
and
\begin{gather*}
\mathcal{K}_{q}^{\vec{n},\vec{\alpha},\beta}=(-1)^{\vert \vec{n} \vert}[-\beta]_{q}^{\left(\vert \vec{n} \vert\right)}q^{-\frac{\vert \vec{n}
\vert}{2}}\left(\prod\limits_{i=1}^{r}\frac{\alpha_i^{n_i}\prod\limits_{j=1}^{n_i}q^{|\vec{n}|_{i}+\beta+j -1}}{\prod\limits_{j=1}^{n_i}\left(\alpha_{i} q^{|\vec{n}|+\beta+j -1}-1\right)}\right)\left(\prod\limits_{i=1}^{r}q^{n_i\sum\limits_{j=i}^{r}n_{j}}\right),
\end{gather*}
where $\vert \vec{n}\vert_{i}=n_{1}+\cdots +n_{i-1}$, $\vert \vec{n}\vert_{1}=0$.
\end{proposition}

\begin{proof} We first start by f\/inding the raising operators. Thus,  we substitute $[s]_{q}^{(k)}$ in~\eqref{NOrthC} by the following f\/inite-dif\/ference
expression
\begin{gather*}
[s]_{q}^{(k)}=\frac{q^{(k-1) /2}}{[k+1]_{q}^{(1)}}\nabla[s+1]_{q}^{(k+1)},
\label{diffs}
\end{gather*}
i.e.,
\begin{gather*}
\sum\limits_{s=0}^{\infty}M_{q,\vec{n}}^{\vec{\alpha},\beta}(s) \nabla
[s+1]_{q}^{(k+1)}\upsilon_{q}^{\alpha_{i}}(s) \bigtriangledown x_{1}(s) =0,
\qquad
0\leq k\leq n_{i}-1,
\qquad
i=1,\ldots,r.
\qquad
\end{gather*}
Then, using summation by parts along with conditions
$\upsilon_{q}^{\alpha_{i},\beta}(-1)=\upsilon_{q}^{\alpha_{i},\beta}(\infty)=0$, we get ~$r$ raising operators
\begin{gather}
\mathcal{D}_{q}^{\alpha_{i},\beta}M_{q,\vec{n}}^{\vec{\alpha},\beta}(s)
=-q^{1/2}M_{q,\vec{n}+\vec{e}_{i}}^{\vec{\alpha}_{i,1/q},\beta-1}(s),
\qquad
i=1,\ldots,r,
\label{ROpqMCharlier}
\end{gather}
where
\begin{gather*}
\mathcal{D}_{q}^{\alpha_{i},\beta}=
\left(\frac{\alpha_{i}x\left(\beta -1\right) q^{|\vec{n}|}}{\left(1-\alpha_{i}q^{|\vec{n}|+\beta -1}\right)\upsilon_{q}^{\alpha_{i}/q ,\beta -1}(s)} \nabla \upsilon_{q}^{\alpha_{i},\beta}(s)\right),\\
\\
\vec{\alpha}_{i,1/q}=(\alpha_{1},\ldots,\alpha_{i}/q,\ldots,\alpha_{r}).
\end{gather*}
Indeed,
\begin{align*}
q^{-\vert \vec{n}\vert -1/2}\mathcal{D}_{q}^{\alpha_{i},\beta}f(s) &=\frac{1}{1-\alpha_{i} q^{|\vec{n}|+\beta -1}}\left[\alpha_{i}q^{\beta -1}\left(x(s)-x(1-\beta)\right)-x(s) \right] f(s)\\
&+\frac{1}{1-\alpha_{i} q^{|\vec{n}|+\beta -1}}x(s)
\bigtriangledown f(s),
\end{align*}
for any function $f(s)$ def\/ined on the discrete variable~$s$. Observe that  $\mathcal{D}_{q}^{\alpha_{i},\beta}$ raises  by 1 the~$i$-th component of the multi-index
$\vec{n}$ in~\eqref{ROpqMCharlier}.

Finally, using the raising operators~\eqref{ROpqMCharlier} in a~recursive way one obtains the
Rodrigues-type formula~\eqref{RFormula}.
\end{proof}

\begin{remark}
Since we are dealing with an AT-system of positive discrete measures \eqref{Measu}, then the $q$-Meixner multiple orthogonal polynomial of the first kind
$M_{q,\vec{n}}^{\vec{\alpha},\beta}(s)$ has exactly $\vert \vec{n}\vert$ dif\/ferent zeros on $\mathbb{R}^{+}$
(see~\cite[Theorem 2.1, pp.~26--27]{arvesu_vanAssche}).
\end{remark}

\section[High-order $q$-dif\/ference equation]{High-order $\boldsymbol{q}$-dif\/ference equation}\label{df-eq}

The strategy that we will follow to deduce the $(r+1)$-order $q$-dif\/ference equation is the foliowing.

{\it First step}. Define an $r$-dimensional subspace $\mathbb{V}$ of polynomials on the variable $x(s)$ of degree at most $\vert
\vec{n}\vert-1$ by means of interpolatory conditions.

{\it Second step}. Find the lowering operator and express its action on $M_{q,\vec{n}}^{\vec{\alpha},\beta}(s)$ as a linear combination of the basis elements of  $\mathbb{V}$.

{\it Third step}. Combine the lowering and the raising operators~\eqref{ROpqMCharlier} to get an $(r+1)$-order $q$-dif\/ference equation
(in the same fashion that~\cite{arvesu-esposito}, \cite{arvesu-ramirez}, and~\cite{lee}).

These steps represent the general features of the algebraic approach we are using in this section, however some ad hoc' computations are needed because of the dependence of the explicit expressions involved in the above steps on the given family of multiple orthogonal polynomials.

\begin{lemma}
Let $\mathbb{V}$ be the linear subspace of polynomials $Q(s)$ on the lattice $x(s)$ of degree at most $\vert
\vec{n}\vert-1$ defined by the following conditions
\begin{gather*}
\sum\limits_{s=0}^{\infty}Q(s) [s]_{q}^{(k)}\upsilon_{q}^{q\alpha_{j},\beta +1}(s)\bigtriangledown x_{1}(s)=0,
\qquad
0\leq k\leq n_{j}-2
\qquad
\text{and}
\qquad
j=1,\ldots,r.
\end{gather*}
Then, the system $\{M_{q,\vec{n}-\vec{e}_{i}}^{\vec{\alpha}_{i,q},\beta +1}(s) \}_{i=1}^{r}$, where
$\vec{\alpha}_{i,q}=(\alpha_{1},\ldots,q\alpha_{i},\ldots,\alpha_{r})$, is a~basis for $\mathbb{V}$.
\label{LI}
\end{lemma}

\begin{proof}

Considering the orthogonality relations
\begin{gather*}
\sum\limits_{s=0}^{\infty}M_{q,\vec{n}-\vec{e}_{i}}^{\vec{\alpha}_{i,q},\beta +1}(s)
[s]_{q}^{(k)}\upsilon_{q}^{q\alpha_{j},\beta +1}(s) \bigtriangledown x_{1}(s) =0,
\qquad
0\leq k\leq n_{j}-2,
\qquad
j=1,\ldots,r,
\end{gather*}
we have that polynomials $M_{q,\vec{n}-\vec{e}_{i}}^{\vec{\alpha}_{i,q},\beta +1}(s)$, $i=1,\ldots,r$, belong to $\mathbb{V}$.

Let proceed by `reductio ad absurdum' and assume that there exists constants $\lambda_{i}$, $i=1,\ldots,r$, such that
\begin{gather*}
\sum\limits_{i=1}^{r}\lambda_{i}M_{q,\vec{n}-\vec{e}_{i}}^{\vec{\alpha}_{i,q},\beta +1}(s) =0,
\qquad
\text{where}
\qquad
\sum\limits_{i=1}^{r}\abs{\lambda_{i}}>0.
\end{gather*}
Then, multiplying the previous equation by $[s]_{q}^{(n_{k}-1)}\upsilon_{q}^{\alpha_{k},\beta }(s)\bigtriangledown
x_{1}(s)$ and then taking summation on~$s$ from $0$ to~$\infty$, one gets
\begin{gather*}
\sum\limits_{i=1}^{r}\lambda_{i}\sum\limits_{s=0}^{\infty}M_{q,\vec{n}-\vec{e}_{i}}^{\vec{\alpha}_{i,q},\beta +1}(s)
[s]_{q}^{(n_{k}-1)}\upsilon_{q}^{\alpha_{k},\beta }(s)\bigtriangledown x_{1}(s) =0.
\end{gather*}
Hence, from relations
\begin{gather}
\sum\limits_{s=0}^{\infty}M_{q,\vec{n}-\vec{e}_{i}}^{\vec{\alpha}_{i,q},\beta +1}(s)
[s]_{q}^{(n_{k}-1)}\upsilon_{q}^{\alpha_{k},\beta }(s) \bigtriangledown x_{1}(s) =c\delta_{i,k},
\qquad
c\in\mathbb{R}\setminus\{0\},
\label{Orthog_Cond}
\end{gather}
we deduce that $\lambda_{k}=0$ for $k=1,\ldots,r$.
Here $\delta_{i,k}$ represents the Kronecker delta symbol.
Therefore, $\{M_{q,\vec{n}-\vec{e}_{i}}^{\vec{\alpha}_{i,q},\beta +1}(s) \}_{i=1}^{r}$ is linearly independent in
$\mathbb{V}$.
Furthermore, we know that any polynomial of $\mathbb{V}$ can be determined with $\vert \vec{n}\vert$
coef\/f\/icients while $(\vert \vec{n}\vert -r)$ linear conditions are imposed on $\mathbb{V}$,
consequently the dimension of $\mathbb{V}$ is at most~$r$.
Hence, the system $\{M_{q,\vec{n}-\vec{e}_{i}}^{\vec{\alpha}_{i,q},\beta +1}(s)\}_{i=1}^{r}$ spans $\mathbb{V}$, which
completes the proof.
\end{proof}

Now we will prove that operator~\eqref{lower} is indeed a~lowering operator for the sequence of $q$-Meixner multiple orthogonal polynomials of the first kind 
$M_{q,\vec{n}}^{\vec{\alpha},\beta +1}(s)$.

\begin{lemma}
There holds the following relation
\begin{gather}
\Delta M_{q,\vec{n}}^{\vec{\alpha},\beta }(s) =\sum\limits_{i=1}^{r}q^{|\vec{n}|-n_{i}+1/2}
\frac{1-\alpha_{i}q^{n_{i}+\beta }}{1-\alpha_{i}q^{|\vec{n}|+\beta }}[n_{i}]_{q}^{(1)}M_{q,\vec{n}-\vec{e}_{i}}^{\vec{\alpha}_{i,q},\beta +1}(s).
\label{Rela_qChar}
\end{gather}
\end{lemma}

\begin{proof}
Using summation by parts we have
\begin{gather}
\sum\limits_{s=0}^{\infty}\Delta M_{q,\vec{n}}^{\vec{\alpha},\beta }(s)[s]_{q}^{(k)}\upsilon_{q}^{q\alpha_{j},\beta +1}(s) \bigtriangledown x_{1}(s)
=-\sum\limits_{s=0}^{\infty}M_{q,\vec{n}}^{\vec{\alpha},\beta }(s)\nabla\big[[s]_{q}^{(k)}\upsilon_{q}^{q\alpha_{j},\beta +1}(s)\big] \bigtriangledown x_{1}(s)
\nonumber
\\
\hphantom{\sum\limits_{s=0}^{\infty}\Delta M_{q,\vec{n}}^{\vec{\alpha},\beta }(s)[s]_{q}^{(k)}\upsilon_{q}^{q\alpha_{j},\beta +1}(s) \bigtriangledown x_{1}(s)}{}
=-\sum\limits_{s=0}^{\infty}M_{q,\vec{n}}^{\vec{\alpha},\beta }(s)\varphi_{j,k}(s)\upsilon_{q}^{\alpha_{j},\beta }(s) \bigtriangledown x_{1}(s),
\label{inte-1}
\end{gather}
where
\begin{gather*}
\varphi_{j,k}(s) =q^{1/2}\left(\frac{q^{\beta}x(s)}{x(\beta)}+1\right)[s]_{q}^{(k)}-q^{-1/2}\frac{x(s)}{\alpha_{j}x(\beta)}[s-1]_{q}^{(k)},
\end{gather*}
is a~polynomial of degree $\leq k+1$ in the variable $x(s)$.
Consequently, from the orthogonality conditions~\eqref{NOrthC} we get
\begin{gather*}
\sum\limits_{s=0}^{\infty}\Delta M_{q,\vec{n}}^{\vec{\alpha},\beta }(s) [s]_{q}^{(k)}\upsilon_{q}^{q\alpha_{j},\beta +1}(s)
\bigtriangledown x_{1}(s)=0,
\qquad
0\leq k\leq n_{j}-2,
\qquad
j=1,\ldots,r.
\end{gather*}
Hence, from Lemma~\ref{LI}, $\Delta M_{q,\vec{n}}^{\vec{\alpha},\beta }(s) \in \mathbb{V}$.
Moreover, $\Delta M_{q,\vec{n}}^{\vec{\alpha},\beta }(s)$ can univocally be expressed as a~linear combination of polynomials
$\{M_{q,\vec{n}-\vec{e}_{i}}^{\vec{\alpha}_{i,q},\beta +1}(s) \}_{i=1}^{r}$, i.e.~
\begin{gather}
\Delta M_{q,\vec{n}}^{\vec{\alpha},\beta }(s) =\sum\limits_{i=1}^{r}\beta_{i}M_{q,\vec{n}-\vec{e}_{i}}^{\vec{\alpha}_{i,q},\beta +1}(s),
\qquad
\sum\limits_{i=1}^{r}\abs{\beta_{i}}>0.
\label{eq-del}
\end{gather}
Multiplying both sides of the equation~\eqref{eq-del} by $[s]_{q}^{(n_{k}-1)}\upsilon_{q}^{q\alpha_{k},\beta +1}(s)
\bigtriangledown x_{1}(s)$ and using relations~\eqref{Orthog_Cond} one has
\begin{multline}
\sum\limits_{s=0}^{\infty}\Delta M_{q,\vec{n}}^{\vec{\alpha},\beta }(s)[s]_{q}^{(n_{k}-1)}\upsilon_{q}^{q\alpha_{k},\beta +1}(s)\bigtriangledown x_{1}(s)
=\sum\limits_{i=1}^{r}\beta_{i}\sum\limits_{s=0}^{\infty}M_{q,\vec{n}-\vec{e}_{i}}^{\vec{\alpha}_{i,q},\beta +1}(s)
[s]_{q}^{(n_{k}-1)}\upsilon_{q}^{q\alpha_{k},\beta +1}(s) \bigtriangledown x_{1}(s)
\\
=\beta_{k}\sum\limits_{s=0}^{\infty}M_{q,\vec{n}-\vec{e}_{k}}^{\vec{\alpha}_{i,q},\beta +1}(s)[s]_{q}^{(n_{k}-1)}\upsilon_{q}^{q\alpha_{k},\beta +1}(s)
\bigtriangledown x_{1}(s).
\label{Ident_I}
\end{multline}
If we replace $[s]_{q}^{(k)}$ by $[s]_{q}^{(n_k-1)}$ in the left-hand side of equation~\eqref{inte-1}, then left-hand
side of equation~\eqref{Ident_I} transforms into relation
\begin{multline}
\sum\limits_{s=0}^{\infty}\Delta M_{q,\vec{n}}^{\vec{\alpha},\beta }(s)[s]_{q}^{(n_{k}-1)}\upsilon_{q}^{q\alpha_{k},\beta +1}(s) \bigtriangledown x_{1}(s)
=-\sum\limits_{s=0}^{\infty}M_{q,\vec{n}}^{\vec{\alpha},\beta }(s)\varphi_{k,n_{k}-1}(s)\upsilon_{q}^{\alpha_{k},\beta }(s) \bigtriangledown x_{1}(s)
\\
=\frac{q^{-1/2}\left(1-\alpha_{k}q^{n_{k}+\beta }\right)}{\alpha_{k}x(\beta)}
\sum\limits_{s=0}^{\infty} M_{q,\vec{n}}^{\vec{\alpha},\beta }(s)[s]_{q}^{(n_{k})}
\upsilon_{q}^{\alpha_{k},\beta }(s)\bigtriangledown x_{1}(s).
\label{eqcha-ra}
\end{multline}
Here we have used that $x(s)[s-1]_{q}^{(n_{k}-1)}=[s]_{q}^{(n_{k})}$ to get
$\varphi_{k,n_{k}-1}(s) =- \frac{q^{-1/2}\left(1-\alpha_{k}q^{n_{k}+\beta }\right)}{\alpha_{k}x(\beta)}[s]_{q}^{(n_{k})}+ \text{lower terms}$.

On the other hand, from~\eqref{ROpqMCharlier} one has that
\begin{gather}
\frac{q^{-1/2}\left(1-\alpha_{k}q^{|\vec{n}|+\beta }\right)}{\alpha_{k}x(\beta)}\upsilon_{q}^{\alpha_{k},\beta}(s) M_{q,\vec{n}}^{\vec{\alpha},\beta}(s) =-q^{|\vec{n}|-1/2}\nabla
\big[\upsilon_{q}^{q\alpha_{k},\beta +1}(s) M_{q,\vec{n}-\vec{e}_{k}}^{\vec{\alpha}_{i,q},\beta +1}(s)\big].
\label{eqcha-ra1}
\end{gather}
Then, by conveniently substituting~\eqref{eqcha-ra1} in the right-hand side of equation~\eqref{eqcha-ra} and using once
more summation by parts, we get
\begin{multline}
\sum\limits_{s=0}^{\infty}\Delta M_{q,\vec{n}}^{\vec{\alpha},\beta}(s)
[s]_{q}^{(n_{k}-1)}\upsilon_{q}^{q\alpha_{k},\beta +1}(s) \bigtriangledown x_{1}(s)\\
 =-q^{|\vec{n}|-1}\frac{1-\alpha_{k}q^{n_{k}+\beta }}{1-\alpha_{k}q^{|\vec{n}|+\beta }}\sum\limits_{s=0}^{\infty}[s]_{q}^{(n_{k})}\nabla \left[\upsilon_{q}^{q\alpha_{k},\beta +1}(s)
M_{q,\vec{n}-\vec{e}_{k}}^{\vec{\alpha}_{i,q},\beta +1}(s) \right] \bigtriangledown x_{1}(s)
\\
 =q^{|\vec{n}|-1}\frac{1-\alpha_{k}q^{n_{k}+\beta }}{1-\alpha_{k}q^{|\vec{n}|+\beta }}\sum\limits_{s=0}^{\infty}M_{q,\vec{n}-\vec{e}_{k}}^{\vec{\alpha}_{i,q},\beta +1}(s) \Delta
\big[[s]_{q}^{(n_{k})}\big]\upsilon_{q}^{q\alpha_{k},\beta +1}(s) \bigtriangledown x_{1}(s).\notag
\end{multline}
Since $\Delta [s]_{q}^{(n_{k})}=q^{3/2-n_{k}}[n_{k}]_{q}^{(1)}[s]_{q}^{(n_{k}-1)}$ we
f\/inally have
\begin{gather*}
\sum\limits_{s=0}^{\infty}\Delta M_{q,\vec{n}}^{\vec{\alpha},\beta}(s)
[s]_{q}^{(n_{k}-1)}\upsilon_{q}^{q\alpha_{k},\beta +1}(s) \bigtriangledown x_{1}(s)
\\
\qquad{}
=q^{|\vec{n}|-n_{k}+1/2}\frac{1-\alpha_{k}q^{n_{k}+\beta }}{1-\alpha_{k}q^{|\vec{n}|+\beta }}[n_{k}]_{q}^{(1)}\sum\limits_{s=0}^{\infty}M_{q,\vec{n}-\vec{e}_{k}}^{\vec{\alpha}_{i,q},\beta +1}(s)
[s]_{q}^{(n_{k}-1)}\upsilon_{q}^{q\alpha_{k},\beta +1}(s) \bigtriangledown x_{1}(s).
\end{gather*}
Therefore, comparing this equation with~\eqref{Ident_I} we obtain the coef\/f\/icients in the expansion~\eqref{eq-del},
i.e.~
\begin{gather*}
\beta_{k}=q^{|\vec{n}|-n_{k}+1/2}\frac{1-\alpha_{k}q^{n_{k}+\beta }}{1-\alpha_{k}q^{|\vec{n}|+\beta }}[n_{k}]_{q}^{(1)},
\end{gather*}
which proves relation~\eqref{Rela_qChar}.
\end{proof}

\begin{theorem}
The $q$-Meixner multiple orthogonal polynomial of the first kind $M_{q,\vec{n}}^{\vec{\alpha},\beta}(s)$ satisfies the following
$(r+1)$-order $q$-difference equation
\begin{gather}
\prod\limits_{i=1}^{r}\mathcal{D}_{q}^{\alpha_{i},\beta}\Delta M_{q,\vec{n}}^{\vec{\alpha},\beta}(s)
=-\sum\limits_{i=1}^{r}q^{\vert \vec{n}\vert -n_{i}+1}\frac{1-\alpha_{i}q^{n_{i}+\beta }}{1-\alpha_{i}q^{|\vec{n}|+\beta }}[n_{i}]_{q}^{(1)}
\prod\limits_{\substack{j=1
\\
j\neq i}}^{r}\mathcal{D}_{q}^{\alpha_{j},\beta}M_{q,\vec{n}}^{\vec{\alpha},\beta}(s).
\label{q-DEquation}
\end{gather}
\end{theorem}

\begin{proof}
Since operators~\eqref{ROpqMCharlier} are commuting, we write
\begin{gather}
\prod\limits_{i=1}^{r}\mathcal{D}_{q}^{\alpha_{i},\beta}
=\left(\prod\limits_{\substack{j=1\\j\neq i}}^{r} \mathcal{D}_{q}^{\alpha_{j},\beta}\right) \mathcal{D}_{q}^{\alpha_{i},\beta},
\label{interm}
\end{gather}
and then using~\eqref{ROpqMCharlier}, by acting on equation~\eqref{Rela_qChar} with the product of
operators~\eqref{interm}, we obtain the following relation
\begin{gather*}
\prod\limits_{i=1}^{r}\mathcal{D}_{q}^{\alpha_{i},\beta}\Delta M_{q,\vec{n}}^{\vec{\alpha},\beta}(s)
 = \sum\limits_{i=1}^{r}q^{|\vec{n}|-n_{i}+1/2}\frac{1-\alpha_{i}q^{n_{i}+\beta }}{1-\alpha_{i}q^{|\vec{n}|+\beta }}[n_{i}]_{q}^{(1)}\left(\prod\limits_{\substack{j=1\\j\neq i}}^{r}\mathcal{D}_{q}^{\alpha_{j},\beta}\right)
\mathcal{D}_{q}^{\alpha_{i},\beta}M_{q,\vec{n}-\vec{e}_{i}}^{\vec{\alpha}_{i,q},\beta +1}(s)
\\
\hphantom{\prod\limits_{i=1}^{r}\mathcal{D}_{q}^{\alpha_{i},\beta}\Delta M_{q,\vec{n}}^{\vec{\alpha},\beta}(s)}{}
 = -\sum\limits_{i=1}^{r}q^{\vert \vec{n}\vert -n_{i}+1}\frac{1-\alpha_{i}q^{n_{i}+\beta }}{1-\alpha_{i}q^{|\vec{n}|+\beta }}[n_{i}
]_{q}^{(1)}\prod\limits_{\substack{j=1
\\
j\neq i}}^{r} \mathcal{D}_{q}^{\alpha_{j},\beta}M_{q,\vec{n}}^{\vec{\alpha},\beta}(s),
\end{gather*}
which proves~\eqref{q-DEquation}.
\end{proof}

\section{Nearest neighbor recurrence relation}\label{rr}

Let us start recalling that for any function $f(s)$ defined on the discrete variable $s$ and a positive integer $n_{i}$ there holds (see Lemma 5.1 from~\cite{arvesu-ramirez})
\begin{equation}
\mathcal{M}_{q,n_{i}}^{\alpha_{i}}\left(x(s) f(s)\right) =q^{-n_{i}+1/2}x\left(n_{i}\right) \left(\alpha_{i}\right)^{-s}\nabla^{n_{i}-1}
\left(\alpha_{i}q^{n_{i}}\right)^{s}f(s) +\dfrac{ x(s) -x\left(n_{i}\right)}{q^{n_{i}} }\mathcal{D}_{q,n_{i}}^{\alpha_{i}}f(s),
\label{eq-lem-meix}
\end{equation}
where difference operator $\mathcal{M}_{q,n_{i}}^{\alpha_{i}}$ is given in \eqref{Op-q-Meix}.

Now, let us proceed with the nearest neighbor recurrence relation.

\begin{theorem}
The $q$-Meixner multiple orthogonal polynomials of the first kind satisfy the following $(r+2)$-term recurrence relation
\begin{align}
x(s) \tilde{M}_{q,\vec{n}}^{\vec{\alpha},\beta}(s) &=c_{\vec{n},k}\tilde{M}_{q,\vec{n}+\vec{e}_{k}}^{\vec{\alpha},\beta}(s)
+b_{\vec{n},k} \tilde{M}_{q,\vec{n}}^{\vec{\alpha},\beta}(s)
\nonumber
\\
&+\sum\limits_{i=1}^{r}x(n_{i})\prod\limits_{j\neq i}^{r}\frac{\alpha_{i}q^{|\vec{n}|}-\alpha_{j}q^{n_{j}}}{\alpha_{i}q^{n_{i}}-\alpha_{j}q^{n_{j}}}\prod\limits_{i=1}^{r}\frac{\alpha_{i}q^{|\vec{n}|+\beta -1}-1}{\alpha_{i}q^{|\vec{n}|+\beta +n_{i}-1}-1}\frac{\alpha_{i}q^{n_{i}}-1}{\alpha_{i}q^{|\vec{n}|+\beta +n_{i}}-1}
\nonumber
\\
&\times\frac{\alpha_{i}q^{|\vec{n}| +n_{i}-1}x(\beta +|\vec{n}|-1)}{\alpha_{i}q^{|\vec{n}|+\beta +n_{i}-1}-1}\frac{1}{\alpha_{i}q^{|\vec{n}|+\beta +n_{i}-2}-1}
\tilde{M}_{q,\vec{n}-\vec{e}_{i}}^{\vec{\alpha},\beta}(s),
\label{q-Rrelation}
\end{align}
where $\tilde{M}_{q,\vec{n}}^{\vec{\alpha},\beta}(s)=\left(\mathcal{K}_{q}^{\vec{n},\vec{\alpha},\beta}\right)^{-1}M_{q,\vec{n}}^{\vec{\alpha},\beta}(s)$ and
\begin{gather*}
b_{\vec{n},k}=\prod\limits_{i=1}^{r}\frac{\alpha_{i}q^{|\vec{n}|+\beta }-1}{\alpha_{i}q^{|\vec{n}|+\beta +n_{i}}-1}\left(\sum\limits_{i=1}^{r}q^{|\vec{n}|_{i}}x(n_{i})\frac{\alpha_{i}q^{\sum\limits_{j=i}^{r}n_{j}}-1}{\alpha_{i}q^{|\vec{n}|+\beta }-1}+\big(q-1\big)q^{|\vec{n}|+\beta}\prod\limits_{i=1}^{r}\frac{\alpha_{i}q^{n_{i}}-1}{\alpha_{i}q^{|\vec{n}|+\beta }-1}\right.
\\
\left.-\big(q-1\big)\prod\limits_{i=1}^{r}x(n_{i})\sum\limits_{i=1}^{r}\prod\limits_{j=1}^{i}\frac{1}{\alpha_{j}q^{|\vec{n}|+\beta }-1}+q^{|\vec{n}|}\frac{(\alpha_{k}q^{n_k+1})x(\beta +|\vec{n}|)}{1-\alpha_{k}q^{|\vec{n}|+\beta+n_k+1}}\right)
\\
-\sum\limits_{i=1}^{r}\prod\limits_{j\neq i}^{r}\frac{\alpha_{i}q^{|\vec{n}|}-\alpha_{j}q^{n_{j}}}{\alpha_{i}q^{n_{i}}-\alpha_{j}q^{n_{j}}}\frac{x(n_{i})\big(\alpha_{i}q^{n_{i}}-1\big)}{\alpha_{i}q^{|\vec{n}|+\beta +n_{i}}-1}\frac{\alpha_{i}q^{\beta + |\vec{n}|+n_{i}-1}}{\alpha_{i}q^{|\vec{n}|+\beta +n_{i}-1}-1},
\end{gather*}
$$
c_{\vec{n},k}=q^{|\vec{n}|-1/2}\prod\limits_{i=1}^{r}
\frac{\alpha_{i}q^{|\vec{n}|+\beta }-1}{\alpha_{i}q^{|\vec{n}|+\beta +n_{i}}-1}\frac{(\alpha_{k}q^{n_k+1})x(\beta +|\vec{n}|)}{\alpha_{k}q^{|\vec{n}|+\beta+n_k+1}-1}.
$$

\end{theorem}

\begin{proof} Consider equation 
\begin{multline}
(\alpha_{k})^{-s}\nabla^{n_{k}+1}(\alpha_{k}q^{n_k+1})^{s} \frac{\Gamma_{q}(\beta +|\vec{n}|+1+ s)}{\Gamma_{q}(\beta +|\vec{n}|+1)\Gamma_{q}(s+1)}\\
=(\alpha_{k})^{-s}\nabla^{n_{k}} \left[q^{-s+1/2}\bigtriangledown \left((\alpha_{k}q^{n_k+1})^{s}
\frac{\Gamma_{q}(\beta +|\vec{n}|+1+ s)}{\Gamma_{q}(\beta +|\vec{n}|+1)\Gamma_{q}(s+1)}\right)\right]
\\
=q^{1/2}(\alpha_{k})^{-s}\nabla^{n_{k}}(\alpha_{k}q^{n_k})^{s}\frac{\Gamma_{q}(\beta +|\vec{n}|+ s)}{\Gamma_{q}(\beta +|\vec{n}|)\Gamma_{q}(s+1)}
\\
+q^{1/2}\frac{\alpha_{k}q^{|\vec{n}|+\beta+n_k+1}-1}{(\alpha_{k}q^{n_k+1})x(\beta +|\vec{n}|)}
(\alpha_{k})^{-s}\nabla^{n_{k}}(\alpha_{k}q^{n_k})^{s} x(s)
\frac{\Gamma_{q}(\beta +|\vec{n}|+ s)}{\Gamma_{q}(\beta +|\vec{n}|)\Gamma_{q}(s+1)},\notag
\end{multline}
which can be rewritten in terms of dif\/ference operators~\eqref{Op-q-Meix} as follows
\begin{multline}
q^{-1/2}\mathcal{M}_{q,n_{k}+1}^{\alpha_{k}}\frac{\Gamma_{q}(\beta +|\vec{n}|+ 1+s)}{\Gamma_{q}(\beta +|\vec{n}|+1)\Gamma_{q}(s+1)}\\
=\mathcal{M}_{q,n_{k}}^{\alpha_{k}}\left(\frac{\Gamma_{q}(\beta +|\vec{n}|+ s)}{\Gamma_{q}(\beta +|\vec{n}|)\Gamma_{q}(s+1)}\left(1+\frac{\alpha_{k}q^{|\vec{n}|+\beta+n_k+1}-1}{(\alpha_{k}q^{n_k+1})x(\beta +|\vec{n}|)}
x(s)\right)\right).
\label{eqsvec}
\end{multline}
Since operators~\eqref{Op-q-Meix} are commuting the multiplication of equation~\eqref{eqsvec} from the left-hand side~by
the product $\left(\prod\limits_{\substack{j=1\\j\neq k}}^{r}\mathcal{M}_{q,n_{j}}^{\alpha_{j}}\right)$ yields
\begin{multline}
\frac{\alpha_{k}q^{|\vec{n}|+\beta+n_k+1}-1}{(\alpha_{k}q^{n_k+1})x(\beta +|\vec{n}|)}
\mathcal{M}_{q,\vec{n}}^{\vec{\alpha}}x(s)
\frac{\Gamma_{q}(\beta +|\vec{n}|+ s)}{\Gamma_{q}(\beta +|\vec{n}|)\Gamma_{q}(s+1)}\\
=q^{-1/2}
\mathcal{M}_{q,\vec{n} +\vec{e}_{k}}^{\vec{\alpha}}\frac{\Gamma_{q}(\beta +|\vec{n}|+1+ s)}{\Gamma_{q}(\beta +|\vec{n}|+1)\Gamma_{q}(s+1)}
-\mathcal{M}_{q,\vec{n}}^{\vec{\alpha}}\frac{\Gamma_{q}(\beta +|\vec{n}|+ s)}{\Gamma_{q}(\beta +|\vec{n}|)\Gamma_{q}(s+1)}.
\label{eq-rr-interm}
\end{multline}

Using recursively relation \eqref{eq-lem-meix}  involving the product of~$r$ dif\/ference operators
$\mathcal{M}_{q,n_{1}}^{\alpha_{1}},\ldots,\mathcal{M}_{q,n_{r}}^{\alpha_{r}}$ acting on the function $f(s)=\frac{\Gamma_{q}(\beta +|\vec{n}|+ s)}{\Gamma_{q}(\beta +|\vec{n}|)\Gamma_{q}(s+1)}$, we have
\begin{align}
&q^{\vert \vec{n}\vert}\mathcal{M}_{q,\vec{n}}^{\vec{\alpha}} x(s)\frac{\Gamma_{q}(\beta +|\vec{n}|+ s)}{\Gamma_{q}(\beta +|\vec{n}|)\Gamma_{q}(s+1)}
\nonumber\\
&=q^{1/2}\sum\limits_{i=1}^{r}\prod\limits_{j\neq i}^{r}\frac{\alpha_{i}q^{|\vec{n}|}-\alpha_{j}q^{n_{j}}}{\alpha_{i}q^{n_{i}}-\alpha_{j}q^{n_{j}}}\frac{x(n_{i})\big(\alpha_{i}q^{n_{i}}-1\big)\big(\alpha_{j}q^{|\vec{n}|+\beta +n_{j}}-1\big)}{\prod\limits_{\nu =1}^{r}\big(\alpha_{\nu}q^{|\vec{n}|+\beta }-1\big)}
\nonumber
\\
&\times\prod\limits_{l=1}^{r}\mathcal{M}_{q,n_{l}-\delta_{l,i}}^{\alpha_{l}}\frac{\Gamma_{q}(\beta +|\vec{n}|+ s)}{\Gamma_{q}(\beta +|\vec{n}|)\Gamma_{q}(s+1)}+\left(\prod\limits_{i=1}^{r}\frac{\alpha_{i}q^{|\vec{n}|+\beta +n_{i}}-1}{\alpha_{i}q^{|\vec{n}|+\beta }-1} x(s) -\sum\limits_{i=1}^{r}q^{|\vec{n}|_{i}}x(n_{i})\right.
\nonumber
\\
&\left.\times\frac{\alpha_{i}q^{\sum\limits_{j=i}^{r}n_{j}}-1}{\alpha_{i}q^{|\vec{n}|+\beta }-1}-\big(q-1\big)q^{|\vec{n}|+\beta}\prod\limits_{i=1}^{r}\frac{\alpha_{i}q^{n_{i}}-1}{\alpha_{i}q^{|\vec{n}|+\beta }-1}+\big(q-1\big)\prod\limits_{i=1}^{r}x(n_{i})\sum\limits_{i=1}^{r}\prod\limits_{j=1}^{i}\frac{1}{\alpha_{j}q^{|\vec{n}|+\beta }-1}\right)
\nonumber\\
&\times \mathcal{M}_{q,\vec{n}}^{\vec{\alpha}}\frac{\Gamma_{q}(\beta +|\vec{n}|+ s)}{\Gamma_{q}(\beta +|\vec{n}|)\Gamma_{q}(s+1)}.
\label{eq-rr-interm1}
\end{align}
Hence, by using expressions~\eqref{eq-rr-interm}, \eqref{eq-rr-interm1} one gets
\begin{align*}
&x(s)\mathcal{M}_{q,\vec{n}}^{\vec{\alpha}}\frac{\Gamma_{q}(\beta +|\vec{n}|+ s)}{\Gamma_{q}(\beta +|\vec{n}|)\Gamma_{q}(s+1)}
\nonumber\\
&=q^{|\vec{n}|-1/2}\prod\limits_{i=1}^{r}\frac{\alpha_{i}q^{|\vec{n}|+\beta }-1}{\alpha_{i}q^{|\vec{n}|+\beta +n_{i}}-1}\frac{(\alpha_{k}q^{n_k+1})x(\beta +|\vec{n}|)}{\alpha_{k}q^{|\vec{n}|+\beta+n_k+1}-1}\mathcal{M}_{q,\vec{n} +\vec{e}_{k}}^{\vec{\alpha}}\frac{\Gamma_{q}(\beta +|\vec{n}|+1+ s)}{\Gamma_{q}(\beta +|\vec{n}|+1)\Gamma_{q}(s+1)}
\\
&+\prod\limits_{i=1}^{r}\frac{\alpha_{i}q^{|\vec{n}|+\beta }-1}{\alpha_{i}q^{|\vec{n}|+\beta +n_{i}}-1}\left(\sum\limits_{i=1}^{r}q^{|\vec{n}|_{i}}x(n_{i})\frac{\alpha_{i}q^{\sum\limits_{j=i}^{r}n_{j}}-1}{\alpha_{i}q^{|\vec{n}|+\beta }-1}+\big(q-1\big)q^{|\vec{n}|+\beta}\prod\limits_{i=1}^{r}\frac{\alpha_{i}q^{n_{i}}-1}{\alpha_{i}q^{|\vec{n}|+\beta }-1}\right.
\\
&\left.-\big(q-1\big)\prod\limits_{i=1}^{r}x(n_{i})\sum\limits_{i=1}^{r}\prod\limits_{j=1}^{i}\frac{1}{\alpha_{j}q^{|\vec{n}|+\beta }-1}+q^{|\vec{n}|}\frac{(\alpha_{k}q^{n_k+1})x(\beta +|\vec{n}|)}{1-\alpha_{k}q^{|\vec{n}|+\beta+n_k+1}}\right)
\\
&\times\mathcal{M}_{q,\vec{n}}^{\vec{\alpha}}\frac{\Gamma_{q}(\beta +|\vec{n}|+ s)}{\Gamma_{q}(\beta +|\vec{n}|)\Gamma_{q}(s+1)}-q^{1/2}\sum\limits_{i=1}^{r}\prod\limits_{j\neq i}^{r}\frac{\alpha_{i}q^{|\vec{n}|}-\alpha_{j}q^{n_{j}}}{\alpha_{i}q^{n_{i}}-\alpha_{j}q^{n_{j}}}\frac{x(n_{i})\big(\alpha_{i}q^{n_{i}}-1\big)}{\alpha_{i}q^{|\vec{n}|+\beta +n_{i}}-1}
\\
&\times\prod\limits_{l=1}^{r}\mathcal{M}_{q,n_{l}-\delta_{l,i}}^{\alpha_{l}}\frac{\Gamma_{q}(\beta +|\vec{n}|+ s)}{\Gamma_{q}(\beta +|\vec{n}|)\Gamma_{q}(s+1)}.
\end{align*}

As a result of the above calculations one has
\begin{align*}
&x(s)\mathcal{M}_{q,\vec{n}}^{\vec{\alpha}}\frac{\Gamma_{q}(\beta +|\vec{n}|+ s)}{\Gamma_{q}(\beta +|\vec{n}|)\Gamma_{q}(s+1)}
\nonumber\\
&=q^{|\vec{n}|-1/2}\prod\limits_{i=1}^{r}\frac{\alpha_{i}q^{|\vec{n}|+\beta }-1}{\alpha_{i}q^{|\vec{n}|+\beta +n_{i}}-1}\frac{(\alpha_{k}q^{n_k+1})x(\beta +|\vec{n}|)}{\alpha_{k}q^{|\vec{n}|+\beta+n_k+1}-1}\mathcal{M}_{q,\vec{n} +\vec{e}_{k}}^{\vec{\alpha}}\frac{\Gamma_{q}(\beta +|\vec{n}|+1+ s)}{\Gamma_{q}(\beta +|\vec{n}|+1)\Gamma_{q}(s+1)}
\\
&+b_{\vec{n},k}\mathcal{M}_{q,\vec{n}}^{\vec{\alpha}}\frac{\Gamma_{q}(\beta +|\vec{n}|+ s)}{\Gamma_{q}(\beta +|\vec{n}|)\Gamma_{q}(s+1)}-q^{1/2}\sum\limits_{i=1}^{r}\prod\limits_{j\neq i}^{r}\frac{\alpha_{i}q^{|\vec{n}|}-\alpha_{j}q^{n_{j}}}{\alpha_{i}q^{n_{i}}-\alpha_{j}q^{n_{j}}}\frac{x(n_{i})\big(\alpha_{i}q^{n_{i}}-1\big)}{\alpha_{i}q^{|\vec{n}|+\beta +n_{i}}-1}
\\
&\times\frac{1}{\alpha_{i}q^{|\vec{n}|+\beta +n_{i}-1}-1}\prod\limits_{l=1}^{r}\mathcal{M}_{q,n_{l}-\delta_{l,i}}^{\alpha_{l}}\frac{\Gamma_{q}(\beta +|\vec{n}|-1+ s)}{\Gamma_{q}(\beta +|\vec{n}|-1)\Gamma_{q}(s+1)}.
\end{align*}

Finally, multiplying from the left both sides of the previous expression~by 
$\frac{\Gamma_{q}(\beta)\Gamma_{q}(s+1)}{\Gamma_{q}(\beta +s)}$ and using Rodrigues-type formula~\eqref{RFormula} we
obtain~\eqref{q-Rrelation}, which completes the proof.
\end{proof}

\section{Concluding remarks}\label{conclu}

In closing, we summarize our findings. We have defined a new family of $q$-Meixner multiple orthogonal polynomials of the first kind and obtained their explicit expression in terms of Rodrigues-type formula \eqref{RFormula}. We have shown that these multiple orthogonal polynomials are common eigenfunctions of two different $(r+1)$-order difference operators given in \eqref{q-DEquation} and \eqref{q-Rrelation}. By taking limit $q\rightarrow 1$ one recovers the corresponding algebraic relations for multiple Meixner polynomials of the first kind \cite{arvesu_vanAssche}. The expressions \eqref{RFormula}, \eqref{q-DEquation}, and \eqref{q-Rrelation} transform into \eqref{Rodrigues-multi}, \eqref{opdi-1}, and \eqref{AR}, respectively.

Our algebraic approach for the nearest neighbor recurrence relation \eqref{q-Rrelation} is purely algebraic and it neither require to introduce type I multiple orthogonality \cite{Assche_neighbor} nor an algebraic Riemann-Hilbert approach. Indeed, the $q$-difference operators involved in the Rodrigues-type formula are the base of the discussed approach.

\subsection*{Acknowledgements}
The research of J.~Arves\'u was partially supported by the research grant MTM2012-36732-C03-01 (Ministerio de
Econom\'{\i}a y Competitividad) of Spain.

\end{document}